\documentclass[10pt]{article}
 \usepackage{mathrsfs}
 \usepackage{float}
 \usepackage{amsfonts}

\usepackage{amssymb,amsmath,amsthm,amsfonts}
\usepackage[dvips]{graphicx}
\usepackage{subfigure}
\usepackage{color}
\usepackage{indentfirst}
\numberwithin{equation}{section}

\newtheorem{theorem}{Theorem}[section]                   

\newtheorem{lemma}{Lemma}[section]
\newtheorem{coro}{Corollary}[section]
\newtheorem{remark}{Remark}[section]

\begin{document}

\title{\bf A New Approach for Solving Delayed Forward and Backward Stochastic Differential Equations
}

\date{}

\author{Tianfu Ma, Juanjuan Xu\footnote{\small Corresponding author. Email:
 juanjuanxu@sdu.edu.cn.} , Huanshui Zhang
\thanks{This work is supported by the National Natural Science Foundation of China under Grants 61633014, 61873332, U1806204, U1701264, 61922051, the foundation for Innovative Research Groups of National Natural Science Foundation of China (61821004). T. Ma, J. Xu and H. Zhang are with the School of Control Science and Engineering,
Shandong University, Jinan, Shandong 250061, China (e-mail:
tianfu\_m@msn.com; hszhang@sdu.edu.cn).}%
}


\maketitle

\begin{abstract}
This paper is concerned with the decoupling of delayed linear forward-backward stochastic differential equations (D-FBSDEs), which is much more involved than the delay-free case due to the infinite dimension caused by the delay. A new approach of `discretization'  is proposed to obtain the explicit solution to the D-FBSDEs. Firstly, we transform the continuous-time D-FBSDEs into the discrete-time form by using discretization. Secondly, we derive the solution of the discrete-time D-FBSDEs by applying backward iterative induction. Finally the explicit solution of the continuous-time D-FBSDEs is obtained by taking the limit to the solution of discrete-time form. The proposed approach can be applied to solve more general FBSDEs with delay, which would provide a complete solution to the stochastic LQ control with time delay.

\vspace{0.3cm}


\vspace{0.3cm}

\end{abstract}

\section{Introduction}
Forward and backward stochastic differential equations (FBSDEs) are widely used in modern society \cite{Fleming,Peng2010}. The related applications can be found in engineering, applied mathematics, finance\cite{Nualart,Allen1}, also they play an important role in stochastic control and differential games\cite{Basar,Yong}.

The mostly motivation arises from the optimal control problems, to be specific, the control problem is reduced to the solvability of Hamiltonian systems which is a special kind of FBSDEs \cite{Wonham1968}. The early work was initiated in \cite{Bismut1978}, through the stochastic Hamiltonian system and the maximum principle. By the stochastic maximum principle, it can be seen that the optimal solution for the control problems depends on solving the corresponding FBSDEs, so the study of solving the FBSDEs is worthful whether in theoretical or practical significance.

In the literature, there are plenty of research focus on the solvability of the FBSDEs \cite{PSG1990,Antonelli,Mao,Ma1994,PSG1995,YJM}. In \cite{PSG1990}, the existence and uniqueness of the adapted solution were shown under the condition of the uniformly Lipschitz continuous. \cite{Antonelli} gave the existence and uniqueness of the solution of a FBSDE, the solution was given under a hypothesis on the Lipschitz constant. \cite{Mao} established the existence and uniqueness of the solution under the condition of non-Lipschitz coefficients. \cite{Ma1994} provided the explicit relationship among the forward and backward variables via the quasilinear partial differential equation by the four-step scheme.  \cite{PSG1995} studied the FBSDEs without the non-degeneracy condition for the forward equation, the existence and uniqueness of the solution was proved. \cite{YJM} made a further study in this direction, and present the method of continuation. After that, the solution of the FBSDEs has been intensively studied, more results can be found in \cite{Ma1999,Peng2000,Wu2005,Wu2013,Tang2013} and the references therein. Moreover, some numerical schemes were proposed in \cite{WYT,WYT2}.

However if the FBSDEs contain time delay, the solvability becomes much more complex. And to our best knowledge, there are few papers studying forward-backward stochastic differential equations with time delay. \cite{Oksendal} studied an infinite horizon system governed by the FBSDEs with delay, sufficient and necessary maximum principles for the optimal control are derived. \cite{Wu2009} considered stochastic optimal control problem of the stochastic delayed system, which can be described by the delayed FBSDEs, through duality method and backward stochastic differential equation (BSDE), the maximum principle was obtained. Due to the process of solving the FBDSEs is more abstract and non intuitive, the addressed problem was the existence and uniqueness of the solution, the explicit solution was not given. Recently, \cite{Juan} considered the FBSDEs with time delay, the explicit solution is given by the modified Riccati equation and two numerical schemes were also given. \cite{Juan1} gave the necessary and sufficient conditions for the existence of a solution to FBSDEs by a Riccati equation. \cite{Hzhang} concerned the problems of linear quadratic regulation control and a class of stochastic system with multiplicative noise and delay, a non-homogeneous relationship was established between the state and the costate of this class of systems. Inspired by these works, we can transform the continuous-times FBSDEs into the discrete-time form and obtain the solution of the discrete-time form, then we can transform the solution of the discrete-time form to the continuous-time form.

In this paper, we consider the solution of the continuous-time linear D-FBSDEs. Firstly, using the method of discretization, we transform the continuous-time D-FBSDEs into the discrete-time form, in this situation, the solution of this discrete-time D-FBSDEs was obtained. Secondly, by taking the limit, we reverse the solution of the discrete-time to the continuous-time, thus we obtain the explicit solution of the continuous-time D-FBSDEs.

This paper is organized into the five parts. Section II presents the problem formulation. The main results are described in the Section III. An application was given in Section IV, and the conclusions are provided into Section V. 

Notation: $R^n$ denotes the family of n-dimensional vectors, the subscript $'$ means the transpose, $\{\Omega,P,\mathcal{F},\{\mathcal{F}_t\}_{t\geq 0}\}$ denotes a complete stochastic basis so that $\mathcal{F}_0$ contains all P-null elements of $\mathcal{F}$. The filtration is generated by the standard one-dimension Brownian motion $\{w(t)\}$. $\hat{x}(t|s)\doteq E[x(t)|\mathcal{F}_s]$ denotes the conditional expectation with respect to the filtration $\mathcal{F}_s$. Then we introduce the following set:

\begin{equation*}
\begin{aligned}
\bar{C}_{[-h,0)}&=\Bigg\{\varphi(t):[-h,0)\to R^m \ \text{is continuous and} \\
&\qquad \sup_{-h\leq t<0}\|\varphi(t)\|<\infty\Bigg\},\\
L_{\mathcal{F}}^{2}(0,T;R^m)&=\Bigg\{\varphi(t)_{t\in[0,T]} \ \text{is an}\  \mathcal{F}_t-\text{adapted stochastic}\\ &\qquad \text{process}\  s.t. \ E\int_{0}^{T}\|\varphi(t)\|^2dt<\infty\Bigg\}.
\end{aligned}
\end{equation*}

\section{Problem Formulation}
In this section, we consider the following linear D-FBSDEs:
\begin{equation}\label{FBSDE}
\begin{aligned}
dx(t)&=\Big\{Ax(t)+BE[p(t)|\mathcal{F}_{t-h}]+CE[q(t)|\mathcal{F}_{t-h}]\Big\}dt\\
&\ \ +\Big\{\bar{A}x(t)+\bar{B}E[p(t)|\mathcal{F}_{t-h}]+\bar{C}E[q(t)|\mathcal{F}_{t-h}]\Big\}dw(t),\\
dp(t)&=-[Dp(t)+\bar{D}q(t)+Qx(t)]dt+q(t)dw(t),\\
x(0)&=x_0,\ p(T)=P(T)x(T),
\end{aligned}
\end{equation}
where $w(t)$ is the one-dimensional standard Brownian motion. $h>0$ is the time delay. $x_0\in R^n$ is initial values,  $A,B,C,D,\bar{A},\bar{B},\bar{C},\bar{D},Q,P(T)$ are constant matrices with compatible dimensions.

It is noticed that linear D-FBSDEs (\ref{FBSDE}) is of a general form and is fully coupled, and the solvability of this D-FBSDEs has wide applications in the stochastic control problems. Due to its strong background like physics, engineering and so on, it received considerable attention, and it has been thoroughly studied over recent decades. In the case of delay-free, i.e.,  $h=0$, the decouple solution of (\ref{FBSDE}) has been given as $p(t)=P(t)x(t)$ where $P(t)$ satisfies a generalized Riccati equation.  However in the delay case, the solution is much more complex.

In this paper, we will present a more intuitive method to solve this D-FBSDEs, and it can be generalized to the more common linear control problems with time delay or multiple input delays.

\section{The Method of Discretization}
In this section, we will give the solution for the continuous-time D-FBSDEs by using the method of discretization, and this will be divided into following steps: first, we will discretize the continuous-time D-FBSDEs (\ref{FBSDE}) to obtain its discrete-time form. Then we are going to obtain the solution of this discrete-time D-FBSDEs. At last, we will reverse the solution to the continuous-time form to obtain the solution of the continuous-time D-FBSDEs (\ref{FBSDE}).

At the beginning, we will make some definitions.  Giving a uniform partition: $0=t_0<t_1<\cdots<t_{N+1}=T$, letting $t_{k+1}=(k+1)\delta,\ t_k=k\delta,\ h=\delta d,\ \Delta w_k=w_{k+1}-w_k$ where $w_k$ is the one-dimensional standard Brownian motion. Then defining $\hat{A}=I+\delta A,\ \hat{B}=\delta B\ ,\hat{C}=\delta C,\ \hat{D}=I+\delta D,\ \hat{P}_{N+1}=P(T),\ \hat{Q}=\delta Q$. $\mathcal{F}_k$ is the natural filtration generated by $w_k$, i.e. $\mathcal{F}_{k}=\sigma\{\Delta w_0,\cdots,\Delta w_k\}$. In order to simplify the symbols, we denote the variables at time $t_k$ as $x_k$, $p_k$ and $q_k$. On the basis, we further define that
\begin{eqnarray*}
A_k&=&\hat{A}+\Delta w_k \bar{A},\\
B_k&=&\hat{B}+\Delta w_k \bar{B},\\
C_k&=&\hat{C}+\Delta w_k \bar{C},\\
D_k&=&\hat{D}+\Delta w_k \bar{D}.
\end{eqnarray*}

Then under these definitions, the discretized D-FBSDEs for (\ref{FBSDE}) will be given.

\begin{lemma}\label{lemma1}
The discretized D-FBSDEs for (\ref{FBSDE}) are given as follows:
\begin{eqnarray}
x_{k+1}&=&A_kx_k+B_kE(p_k|\mathcal{F}_{k-d-1})+\frac{1}{\delta}C_kE(\Delta w_kp_k|\mathcal{F}_{k-d-1}),\label{FSDE}\\
p_{k-1}&=&E(D_kp_k|\mathcal{F}_{k-1})+\hat{Q}x_k\label{BSDE},\\
p_N&=&\hat{P}_{N+1}x_{N+1}.\label{tcon}
\end{eqnarray}
\end{lemma}

\begin{proof}
First, we will illustrate how to discretize the BSDE. Like in \cite{Kailath}, we can rewrite BSDE as
\begin{equation}\label{p_k}
\begin{aligned}
p_k-p_{k-1}&=-\delta[Dp_k+\bar{D}q_k+Qx_k]+q_k\Delta w_k,
\end{aligned}
\end{equation}
By multiplying $\Delta w_k$ on both sides, we can obtain that
\begin{equation*}
\begin{aligned}
{\Delta w_k p_k-\Delta w_kp_{k-1}}&=-\delta\Delta w_k [Dp_k+\bar{D}q_k+Qx_k]+\delta q_k,
\end{aligned}
\end{equation*}
and then taking the conditional expectation with respect to $\mathcal{F}_{k-1}$ on both sides, we will obtain that
\begin{equation*}
\begin{aligned}
E({\Delta w_k p_k-\Delta w_k p_{k-1}}|\mathcal{F}_{k-1})=-E[\delta\Delta w_k(Dp_k+\bar{D}q_k+Qx_k)|\mathcal{F}_{k-1}]+\delta q_k.
\end{aligned}
\end{equation*}
that is,
\begin{equation*}
\begin{aligned}
E[\Delta w_k p_k|\mathcal{F}_{k-1}]=-E[\delta\Delta w_k(Dp_k+\bar{D}q_k+Qx_k)|\mathcal{F}_{k-1}]+\delta q_k.
\end{aligned}
\end{equation*}
Noting that the second term in the above equation is higher order of $\delta$ due to the fact that $E[\Delta w_k^2]=\delta$, it is derived by omitting this higher order term that
$q_k=\displaystyle \frac{1}{\delta}E(\Delta w_kp_k|\mathcal{F}_{k-1})$.

Then inserting $q_k=\displaystyle \frac{1}{\delta}E(\Delta w_k p_k|\mathcal{F}_{k-1})$ in (\ref{p_k}), and taking the conditional expectation with respect to $\mathcal{F}_{k-1}$ on both sides, we have the discretization of the BSDE given by
\begin{equation*}
\begin{aligned}
p_{k-1}=E[(I+\delta D+\Delta w_k \bar D)p_k|\mathcal{F}_{k-1}]+\delta Q x_k,
\end{aligned}
\end{equation*}
this is exactly (\ref{BSDE}). So we finish the discretization of BSDE.

Next we will discretize the forward stochastic differential equation (FSDE). Following the same procedure, we rewrite FSDE in (\ref{FBSDE}) as follows:
\begin{align*}
x_{k+1}-x_k&=\delta\Big[Ax_k+BE(p_k|\mathcal{F}_{k-d-1})+CE(q_k|\mathcal{F}_{k-d-1})\Big]\\
&+\Big[\bar{A}x_k+\bar{B}E(p_k|\mathcal{F}_{k-d-1})+\bar{C}E(q_k|\mathcal{F}_{k-d-1})\Big]\Delta w_k.
\end{align*}
then the above equation can be rewritten as follows:
\begin{equation}\label{xk+1}
\begin{aligned}
x_{k+1}&=(I+\delta A+\Delta w_k \bar{A})x_{k}+(\delta B+\Delta w_{k} \bar{B})E(p_{k}|\mathcal{F}_{k-d-1})\\
&\ +(\delta C+\Delta w_{k} \bar{C})E(q_{k}|\mathcal{F}_{k-d-1}).
\end{aligned}
\end{equation}

Inserting $q_k=\displaystyle \frac{1}{\delta}E(\Delta w_kp_k|\mathcal{F}_{k-1})$ into (\ref{xk+1}), we can rewrite the last term of the (\ref{xk+1}) as
\begin{eqnarray*}
(\delta C+\Delta w_k \bar{C})E(q_k|\mathcal{F}_{k-d-1})=\frac{1}{\delta}(\delta C+ \Delta w_k \bar{C})E(\Delta w_k p_k|\mathcal{F}_{k-d-1}).
\end{eqnarray*}

We thus obtain the discretized FSDEs as:
\begin{eqnarray*}
\begin{aligned}
x_{k+1}&=(I+\delta A+\Delta w_k \bar{A})x_k+(\delta B+\Delta w_k \bar{B})E(p_k|\mathcal{F}_{k-d-1})\\
&\ +\frac{1}{\delta}(\delta C+ \Delta w_k \bar{C})E(\Delta w_k p_k|\mathcal{F}_{k-d-1})\\
&= A_kx_k+B_kE(p_k|\mathcal{F}_{k-d-1})+\frac{1}{\delta}C_kE(\Delta w_kp_k|\mathcal{F}_{k-d-1}).
\end{aligned}
\end{eqnarray*}

This is exactly (\ref{FSDE}). So we finish the discretization of FSDE in (\ref{FBSDE}). This completes the proof.
\end{proof}

The second step is to solve the discrete-time D-FBSDEs (\ref{FSDE})-(\ref{tcon}), so the solution for the discrete-time D-FBSDEs (\ref{FSDE})-(\ref{tcon}) will be shown.

\begin{lemma}\label{lemma2}
D-FBSDEs (\ref{FSDE})-(\ref{tcon}) are uniquely solvable if the Riccati equations (\ref{Pk})-(\ref{Gammak}) admits a solution such that the matrix $\Gamma_k$ is invertible, and the solution for the discrete-time D-FBSDEs are given as follows for $k\geq d$:
\begin{eqnarray}
x_k&=&A_{k-1}x_{k-1}+M_{k-1}E(x_{k-1}|\mathcal{F}_{k-d-2}),\label{DFSDE}\\
p_{k-1}&=&\hat{P}_kx_k-\sum_{i=0}^{d}\hat{P}_{k}^{k+i}E(x_k|\mathcal{F}_{k-d+i-1}),\label{DBSDE}\\
p_{N}&=&\hat{P}_{N+1}x_{N+1},
\end{eqnarray}
where $\hat{P}_k,\hat{P}_{k}^{k+i}$ satisfies:
\begin{align}
\hat{P}_k&=\hat{D}\hat{P}_{k+1}\hat{A}+\delta\bar{D}\hat{P}_{k+1}\bar{A}-\hat{D}\hat{P}_{k+1}^{k+d+1}\hat{A}+\hat{Q},\label{Pk}\\
\hat{P}_{k}^{k+i}&=\hat{D}\hat{P}_{k+1}^{k+i}\hat{A},\\
\hat{P}_{k}^{k}&=-E[(D_{k}\hat{P}_{k+1}-\sum_{i=0}^{d}\hat{D}\hat{P}_{k+1}^{k+1+i})M_{k}],\label{Pkk}\\
M_{k-1}&=\left[
\begin{array}{cc}
\displaystyle B_{k-1}(\hat{P}_k-\sum_{i=0}^{d}\hat{P}_{k}^{k+i}) & \displaystyle \frac{1}{\delta}C_{k-1}\hat{P}_{k} \\
\end{array}
\right]
\Gamma_{k}^{-1}
\left[
   \begin{array}{c}
     \hat{A} \\
     \delta\bar{A} \\
   \end{array}
 \right],\label{Mk}\\
\Gamma_{k}&=
 \left[
  \begin{array}{cc}
    \displaystyle I-\hat{B}(\hat{P}_k-\sum_{i=0}^{d}\hat{P}_{k}^{k+i}) & \displaystyle-\frac{1}{\delta}\hat{C}\hat{P}_{k} \\
    \displaystyle -\delta\bar{B}(\hat{P}_k-\sum_{i=0}^{d}\hat{P}_{k}^{k+i}) & I-\bar{C}\hat{P}_{k} \\
  \end{array}
\right]\label{Gammak},
\end{align}
with the terminal value given by $\hat{P}_{N+1}$. $\hat{P}_{N+1+i}=0,\ \hat{P}_{N+1}^{j}=0$ for any $i,j>0$. In addition, $\hat{P}_{k}^{s}=0$ if $s>N$.
\end{lemma}

\begin{proof}
We complete the proof by using the induction method. We will start the proof at $k=N$ first. For $k=N$, using the terminal condition (\ref{tcon}) and the FSDE (\ref{FSDE}), it yields that
\begin{equation}\label{terminal}
\begin{aligned}
x_{N+1}&=A_Nx_N+B_NE(p_N|\mathcal{F}_{N-d-1})+\frac{1}{\delta}C_NE(\Delta w_Np_N|\mathcal{F}_{N-d-1})\\
&=A_Nx_N+B_NE(\hat{P}_{N+1}x_{N+1}|\mathcal{F}_{N-d-1})\\
&\ +\frac{1}{\delta}C_NE(\Delta w_N\hat{P}_{N+1}x_{N+1}|\mathcal{F}_{N-d-1}).
\end{aligned}
\end{equation}

Taking the conditional expectation with respect to $\mathcal{F}_{N-d-1}$ on both sides of (\ref{terminal}), we have
\begin{equation*}
\begin{aligned}
E(x_{N+1}|\mathcal{F}_{N-d-1})&=\hat{A}E(x_N|\mathcal{F}_{N-d-1})+\hat{B}\hat{P}_{N+1}E(x_{N+1}|\mathcal{F}_{N-d-1})\\
&\ +\frac{1}{\delta}\hat{C}\hat{P}_{N+1}E(\Delta w_N x_{N+1}|\mathcal{F}_{N-d-1}),
\end{aligned}
\end{equation*}
which generates
\begin{equation}\label{1}
\begin{aligned}
&(I-\hat{B}\hat{P}_{N+1})E(x_{N+1}|\mathcal{F}_{N-d-1})\\
&=\hat{A}E(x_N|\mathcal{F}_{N-d-1})+\frac{1}{\delta}\hat{C}\hat{P}_{N+1}E(\Delta w_N x_{N+1}|\mathcal{F}_{N-d-1}).
\end{aligned}
\end{equation}
Then by multiplying $\Delta w_N$ on both sides of (\ref{terminal}) and taking the conditional expectation with respect to $\mathcal{F}_{N-d-1}$, we can obtain the following relationship:
\begin{equation*}
\begin{aligned}
&E(\Delta w_Nx_{N+1}|\mathcal{F}_{N-d-1})\\
&=E(\Delta w_NA_Nx_N|\mathcal{F}_{N-d-1})+E\big[\Delta w_NB_NE(\hat{P}_{N+1}x_{N+1}|\mathcal{F}_{N-d-1})|\mathcal{F}_{N-d-1}\big]\\
&\ +\frac{1}{\delta}E\big[\Delta w_NC_NE(\Delta w_N\hat{P}_{N+1}x_{N+1}|\mathcal{F}_{N-d-1})|\mathcal{F}_{N-d-1}\big],\\
\end{aligned}
\end{equation*}
we thus have
\begin{equation*}
\begin{aligned}
&E(\Delta w_Nx_{N+1}|\mathcal{F}_{N-d-1})\\
&=\delta \bar{A}E(x_N|\mathcal{F}_{N-d-1})+\delta\bar{B}\hat{P}_{N+1}E(x_{N+1}|\mathcal{F}_{N-d-1})\\
&\ +\bar{C}\hat{P}_{N+1}E(\Delta w_N x_{N+1}|\mathcal{F}_{N-d-1}),
\end{aligned}
\end{equation*}
which generates
\begin{equation}\label{2}
\begin{aligned}
&(I-\bar{C}\hat{P}_{N+1})E(\Delta w_N x_{N+1}|\mathcal{F}_{N-d-1})\\
&=\delta\bar{A}E(x_N|\mathcal{F}_{N-d-1})+\delta\bar{B}\hat{P}_{N+1}E(x_{N+1}|\mathcal{F}_{N-d-1}).
\end{aligned}
\end{equation}
Then combining (\ref{1}) and (\ref{2}), we have
\begin{equation*}
\begin{aligned}
&\left[
  \begin{array}{cc}
    I-\hat{B}\hat{P}_{N+1} & \displaystyle-\frac{1}{\delta}\hat{C}\hat{P}_{N+1} \\
    -\delta\bar{B}\hat{P}_{N+1} & I-\bar{C}\hat{P}_{N+1} \\
  \end{array}
\right]
\left[
  \begin{array}{c}
    E(x_{N+1}|\mathcal{F}_{N-d-1}) \\
    E(\Delta w_Nx_{N+1}|\mathcal{F}_{N-d-1}) \\
  \end{array}
\right]
\\
&=\left[
   \begin{array}{c}
     \hat{A} \\
     \delta\bar{A} \\
   \end{array}
 \right]E(x_N|\mathcal{F}_{N-d-1}).
 \end{aligned}
\end{equation*}
This gives that
\begin{equation}\label{3}
\begin{aligned}
&\left[
  \begin{array}{c}
    E(x_{N+1}|\mathcal{F}_{N-d-1}) \\
    E(\Delta w_Nx_{N+1}|\mathcal{F}_{N-d-1}) \\
  \end{array}
\right]=\Gamma_{N+1}^{-1}
\left[
   \begin{array}{c}
     \hat{A} \\
     \delta\bar{A} \\
   \end{array}
 \right]E(x_N|\mathcal{F}_{N-d-1}),
 \end{aligned}
\end{equation}
where
\begin{equation*}
\begin{aligned}
\Gamma_{N+1}=
&\left[
  \begin{array}{cc}
    I-\hat{B}\hat{P}_{N+1} & \displaystyle-\frac{1}{\delta}\hat{C}\hat{P}_{N+1} \\
    -\delta\bar{B}\hat{P}_{N+1} & I-\bar{C}\hat{P}_{N+1} \\
  \end{array}
\right].
 \end{aligned}
\end{equation*}

Introducing (\ref{3}) into (\ref{terminal}), we have
\begin{equation}\label{x_N+1}
\begin{aligned}
x_{N+1}&=A_Nx_N
+\left[
\begin{array}{cc}
B_N\hat{P}_{N+1} & \displaystyle \frac{1}{\delta}C_N\hat{P}_{N+1} \\
\end{array}
\right]
\left[
  \begin{array}{c}
    E(x_{N+1}|\mathcal{F}_{N-d-1}) \\
    E(\Delta w_Nx_{N+1}|\mathcal{F}_{N-d-1}) \\
  \end{array}
\right]\\
&=A_Nx_N+\left[
\begin{array}{cc}
B_N\hat{P}_{N+1} & \displaystyle \frac{1}{\delta}C_N\hat{P}_{N+1} \\
\end{array}
\right]
\Gamma_{N+1}^{-1}
\left[
   \begin{array}{c}
     \hat{A} \\
     \delta\bar{A} \\
   \end{array}
 \right]E(x_N|\mathcal{F}_{N-d-1})\\
&= A_Nx_N+M_NE(x_N|\mathcal{F}_{N-d-1}),
\end{aligned}
\end{equation}
where
\begin{align*}
M_N=\left[
\begin{array}{cc}
B_N\hat{P}_{N+1} & \displaystyle \frac{1}{\delta}C_N\hat{P}_{N+1} \\
\end{array}
\right]
\Gamma_{N+1}^{-1}
\left[
   \begin{array}{c}
     \hat{A} \\
     \delta\bar{A} \\
   \end{array}
 \right].
\end{align*}

Therefore, we proved that (\ref{DFSDE}) is true for $k=N$, Next we will consider BSDE (\ref{DBSDE}) is true for $k=N$.

Combining the terminal condition (\ref{tcon}) and (\ref{x_N+1}) and inserting to (\ref{BSDE}), we have the following relationship:
\begin{equation*}
\begin{aligned}
p_{N-1}&=E(D_Np_N|\mathcal{F}_{N-1})+\hat{Q}x_N\\
&=E(D_N\hat{P}_{N+1}x_{N+1}|\mathcal{F}_{N-1})+\hat{Q}x_N\\
&=E\Big\{D_N\hat{P}_{N+1}\big[A_Nx_N+M_NE(x_N|\mathcal{F}_{N-d-1})\big]|\mathcal{F}_{N-1}\Big\}+\hat{Q}x_N\\
&=E(D_N\hat{P}_{N+1}A_Nx_N|\mathcal{F}_{N-1})+\hat{Q}x_N+E\Big[D_N\hat{P}_{N+1}M_NE(x_N|\mathcal{F}_{N-d-1})|\mathcal{F}_{N-1}\Big]\\
&=(\hat{D}\hat{P}_{N+1}\hat{A}+\delta\bar{D}\hat{P}_{N+1}\bar{A}+\hat{Q})x_N+E(D_N\hat{P}_{N+1}M_N)E(x_N|\mathcal{F}_{N-d-1})\\
&\doteq \hat{P}_Nx_N-\hat{P}_{N}^{N}E(x_N|\mathcal{F}_{N-d-1}),
\end{aligned}
\end{equation*}
where
\begin{equation*}
\begin{aligned}
\hat{P}_N&=\hat{D}\hat{P}_{N+1}\hat{A}+\delta\bar{D}\hat{P}_{N+1}\bar{A}+\hat{Q},\\
\hat{P}_{N}^{N}&=-E(D_N\hat{P}_{N+1}M_N).
\end{aligned}
\end{equation*}

Thus we have verified (\ref{DFSDE})-(\ref{DBSDE}) are true for $k=N$. In order to complete the induction proof, we take any $s$ with $d<s<N$, and assuming that $x_k$ and $p_{k-1}$ are as (\ref{DFSDE})-(\ref{DBSDE}) for all $k\geq s+1$, then we will show that these conditions are also hold for the situation of $k=s$.

From (\ref{BSDE}) at $s$, (\ref{DBSDE}) with $k=s+1$, and (\ref{DFSDE}) with $k=s+1$, one has
\begin{equation*}
\begin{aligned}
p_{s-1}&=E(D_sp_s|\mathcal{F}_{s-1})+\hat{Q}x_s\\
&=E\Big\{D_s\big[\hat{P}_{s+1}x_{s+1}-\sum_{i=0}^{d}\hat{P}_{s+1}^{s+i+1}E(x_{s+1}|\mathcal{F}_{s-d+i})\big]|\mathcal{F}_{s-1}\Big\}+\hat{Q}x_s.\\
\end{aligned}
\end{equation*}

Noticing that $\hat{P}_{k}^{s}=0$ for any $s>N$, then through the simple calculation, we can obtain that
\begin{equation*}
\begin{aligned}
p_{s-1}&=(\hat{D}\hat{P}_{s+1}\hat{A}+\delta\bar{D}\hat{P}_{s+1}\bar{A}-\hat{D}\hat{P}_{s+1}^{s+d+1}\hat{A}+\hat{Q})x_s+E\big[(D_{s}\hat{P}_{s+1}\\
&\ -\sum_{i=0}^{d}D\hat{P}_{s+1}^{s+i})M_{s}\big]E(x_s|\mathcal{F}_{s-d-1})-\sum_{i=0}^{d}\hat{D}\hat{P}_{s+1}^{s+i+1}\hat{A}E(x_s|\mathcal{F}_{s-d+i})\\
& =\hat{P}_sx_s-\sum_{i=0}^{d}\hat{P}_{s}^{s+i}E(x_s|\mathcal{F}_{s-d+i-1}),
\end{aligned}
\end{equation*}
where the coefficients are given via (\ref{Pk})-(\ref{Pkk}).

Similarly, from (\ref{FSDE}) at $s$, (\ref{DBSDE}) with $k=s+1$, we have
\begin{equation*}
\begin{aligned}
x_{s+1}&=A_sx_s+B_sE(p_s|\mathcal{F}_{s-d-1})+\frac{1}{\delta}C_sE(\Delta w_sp_s|\mathcal{F}_{s-d-1}),\\
&=A_sx_s+B_sE[\hat{P}_{s+1}x_{s+1}-\sum_{i=0}^{d}\hat{P}_{s+1}^{s+i+1}E(x_{s+1}|\mathcal{F}_{s-d+i})|\mathcal{F}_{s-d-1}]\\
&\ +\frac{1}{\delta}C_sE\Big\{\Delta w_s[\hat{P}_{s+1}x_{s+1}-\sum_{i=0}^{d}\hat{P}_{s+1}^{s+i+1}E(x_{s+1}|\mathcal{F}_{s-d+i})]|\mathcal{F}_{s-d-1}\Big\}\\
&=A_{s}x_{s}+B_{s}(\hat{P}_{s+1}-\sum_{i=0}^{d}\hat{P}_{s+1}^{s+i+1})E(x_{s+1}|\mathcal{F}_{s-d+i})\\
&\ +\frac{1}{\delta}C_{s}\hat{P}_{s+1}E(\Delta w_{s}x_{s+1}|\mathcal{F}_{s-d+i}).
\end{aligned}
\end{equation*}

Using the similar derivation to (\ref{terminal}), (\ref{1}) and (\ref{2}) at $k=N$, we can obtain that
\begin{equation*}
\begin{aligned}
&\left[
  \begin{array}{cc}
    \displaystyle I-\hat{B}(\hat{P}_{s+1}-\sum_{i=0}^{d}\hat{P}_{s+1}^{s+i+1}) & \displaystyle-\frac{1}{\delta}\hat{C}\hat{P}_{s+1} \\
    \displaystyle -\delta\bar{B}(\hat{P}_{s+1}-\sum_{i=0}^{d}\hat{P}_{s+1}^{s+i+1}) & I-\bar{C}\hat{P}_{s+1} \\
  \end{array}
\right]
\left[
  \begin{array}{c}
    E(x_{s+1}|\mathcal{F}_{s-d-1}) \\
    E(\Delta w_sx_{s+1}|\mathcal{F}_{s-d-1}) \\
  \end{array}
\right]
\\
&=\left[
   \begin{array}{c}
     A \\
     \delta\bar{A} \\
   \end{array}
 \right]E(x_s|\mathcal{F}_{s-d-1}),
 \end{aligned}
\end{equation*}
which gives that
\begin{equation*}
\begin{aligned}
x_{s+1}&=A_sx_s
+\left[
\begin{array}{cc}
\displaystyle B_s(\hat{P}_{s+1}-\sum_{i=0}^{d}\hat{P}_{s+1}^{s+i+1}) & \displaystyle \frac{1}{\delta}C_s\hat{P}_{s+1} \\
\end{array}
\right]\\
&\ \times
\left[
  \begin{array}{c}
    E(x_{s+1}|\mathcal{F}_{s-d-1}) \\
    E(\Delta w_sx_{s+1}|\mathcal{F}_{s-d-1}) \\
  \end{array}
\right]\\
&=A_sx_s+\left[
\begin{array}{cc}
\displaystyle B_s(\hat{P}_{s+1}-\sum_{i=0}^{d}\hat{P}_{s+1}^{s+i+1}) & \displaystyle \frac{1}{\delta}C_s\hat{P}_{s+1} \\
\end{array}
\right]\\
&\ \times
\Gamma_{s+1}^{-1}
\left[
   \begin{array}{c}
     A \\
     \delta\bar{A} \\
   \end{array}
 \right]E(x_s|\mathcal{F}_{s-d-1})\\
&= A_sx_s+M_sE(x_s|\mathcal{F}_{s-d-1}),
\end{aligned}
\end{equation*}
where the coefficients are given by (\ref{Mk})-(\ref{Gammak}). Hence, (\ref{DFSDE})-(\ref{Gammak}) are true for $s=k$. Then the proof is completed by the inductive approach.
\end{proof}

\begin{remark}
It is noted that the expectation in $\hat{P}_{k}^{k}$ can be calculated easily, that is
\begin{equation*}
\begin{aligned}
\hat{P}_{k}^{k}&=-E\big[(D_{k}\hat{P}_{k+1}-\sum_{i=0}^{d}D\hat{P}_{k+1}^{k+i})M_{k}\big]=
\displaystyle -(\alpha_1+\alpha_2)
\end{aligned}
\end{equation*}
where
\begin{align*}
\alpha_1&=(\hat{D}\hat{P}_{k+1}\hat{B}+\delta\bar{D}\hat{P}_{k+1}\bar{B}-\sum_{i=0}^{d}\hat{D}\hat{P}_{k+1}^{k+i+1}\hat{B})(\hat{P}_k-\sum_{i=0}^{d}\hat{P}_{k}^{k+i})\Gamma_{k+1}^{-1}\hat{A},\\
\alpha_2&=\frac{1}{\delta}(\hat{D}\hat{P}_{k+1}\hat{C}+\delta\bar{D}\hat{P}_{k+1}\bar{C}-\sum_{i=0}^{d}\hat{D}\hat{P}_{k+1}^{k+i+1}\hat{C})P_k\Gamma_{k+1}^{-1}\bar{A}.
\end{align*}
\end{remark}

Based on Lemma \ref{lemma2}, it is clearly to find that the solution to the D-FBSDEs depends on the solution to Riccati equations (\ref{Pk})-(\ref{Gammak}), so we can infer that the solution to the coninuous-time D-FBSDEs also relies on the corresponding
Riccati equations.

We are now in the position to give the solution of the continuous-time D-FBSDEs. Using the definition of $A_k,B_k,C_k$ and $D_k$, then letting $\delta\to 0$, we can reverse the equation (\ref{DFSDE})-(\ref{DBSDE}) to the continuous-time form, so the main result of this paper is present in the following theorem.
\begin{theorem}\label{theorem1}
{D-FBSDEs (\ref{FBSDE}) are uniquely solvable if the Riccati equations (\ref{P(t)})-(\ref{P(t,s)}) admits a solution such that the matrix $I-\bar{C}P(t)$ is invertible, and the solution of the continuous-time D-FBSDEs (\ref{FBSDE}) is given as follows for $t\geq h$:}
\begin{align}
p(t)&=P(t)x(t)-\int_{0}^{h}P(t,t+\theta)E[x(t)|\mathcal{F}_{t+\theta-h}]d\theta,\  h\leq t \leq T, \label{pt}\\
q(t)&=P(t)\Big\{\bar{A}x(t)+\bar{B}E[p(t)|\mathcal{F}_{t-h}]+\bar{C}E[q(t)|\mathcal{F}_{t-h}]\Big\},\label{qt}\\
\nonumber dx(t)&=\Big\{Ax(t)+\Big[BS(t)+CP(t)[I-\bar{C}P(t)]^{-1}[\bar{A}+\bar{B}S(t)]\Big]\\
 &\ \times E[x(t)|\mathcal{F}_{t-h}]\Big\}dt+\Big\{\bar{A}x(t)+\Big[\bar{B}S(t)+\bar{C}P(t)[I-\bar{C}P(t)]^{-1}\label{x(t)}\\
\nonumber &\ \times[\bar{A}+\bar{B}S(t)]\Big]E[x(t)|\mathcal{F}_{t-h}]\Big\}dw(t).
\end{align}
where $P(t),P(t,s)$ satifies
\begin{eqnarray}
-\dot{P}(t)&=&DP(t)+P(t)A+\bar{D} P(t)\bar{A}+Q-P(t,t+h), \label{P(t)}\\
\nonumber P(t,t)&=&-\Big[(S(t)B+\bar{D}P(t)\bar{B})S(t)+[I-\bar{C}P(t)]^{-1}(S(t)C\\
&&+\bar{D}P(t)\bar{C})P(t)[\bar{A}+\bar{B}S(t)]\Big],\label{P(t,t)}\\
P(t,s)&=&e^{D(s-t)}P(s,s)e^{A(s-t)},\label{P(t,s)}
\end{eqnarray}

with the terminal value $P(T)$. $P(T,T+\theta)=0$ for any $\theta\geq 0$, $\displaystyle S(t)=P(t)-\int_{0}^{h}P(t,t+\theta)d\theta$ and the related matrices are define as (\ref{FBSDE}).

\end{theorem}

\begin{proof}
First, we will verify (\ref{pt}). Defining the limitation of $\hat P_k$ and $\hat{P}_{k}^{k+i}$ are as $P(t)$ and $P(t,t+\theta)$. and considering the discrete-time form (\ref{DBSDE}):
\begin{equation*}
\begin{aligned}
p_{k-1}&=&\hat{P}_kx_k-\sum_{i=0}^{d}\hat{P}_{k}^{k+i}E(x_k|\mathcal{F}_{k-d+i-1}),
\end{aligned}
\end{equation*}
by taking the limit on both sides we have
\begin{align*}
\displaystyle \nonumber p(t)&=P(t)x(t)-\int_{0}^{h}P(t,t+\theta)E[x(t)|\mathcal{F}_{t+\theta-h}]d\theta,
\end{align*}
which is (\ref{pt}). The proof of the corresponding Riccati equations will be given later.

Next we will verify (\ref{qt}). Using the definition $q_k=\displaystyle \frac{1}{\delta}E(\Delta w_kp_k|\mathcal{F}_{k-1})$ and (\ref{DBSDE}), we can obtain that
\begin{equation*}
\begin{aligned}
q_k&=\frac{1}{\delta}E(\Delta w_kp_k|\mathcal{F}_{k-1})\\
&=\frac{1}{\delta}E\Big\{\Delta w_k\big[\hat{P}_{k+1}x_{k+1}+\sum_{i=0}^{d-1}\hat{P}_{k+1}^{k+1+i}E(x_{k+1}|\mathcal{F}_{k-d+i})\big]|\mathcal{F}_{k-1}\Big\}.
\end{aligned}
\end{equation*}
Inserting (\ref{xk+1}), we have
\begin{equation*}
\begin{aligned}
q_k&=\frac{1}{\delta}E\Big\{\Delta w_k\big\{\hat{P}_{k+1}\big[(I+\delta A+\Delta w_k \bar{A})x_{k}+(\delta B+\Delta w_{k} \bar{B})E(p_{k}|\mathcal{F}_{k-d-1})\\
&\ +(\delta C+\Delta w_{k} \bar{C})E(q_{k}|\mathcal{F}_{k-d-1})\big]+\sum_{i=0}^{d-1}\hat{P}_{k+1}^{k+1+i}E(x_{k+1}|\mathcal{F}_{k-d+i})\big\}|\mathcal{F}_{k-1}\Big\}.
\end{aligned}
\end{equation*}
Through the simple calculation, we thus have
\begin{equation*}
\begin{aligned}
q_k=\hat{P}_{k+1}\Big[\bar{A}x_k+\bar{B}E(p_k|\mathcal{F}_{k-d-1})+\bar{C}E(q_k|\mathcal{F}_{k-d-1})\Big].
\end{aligned}
\end{equation*}
Then taking the limit on both sides, we can obtain the continuous-time form, which is (\ref{qt}).

Next we will show the process of proving (\ref{x(t)}). Inserting $\hat{B}=\delta B,\hat{C}=\delta C$ in $\Gamma_k$, then defining $\displaystyle S_k=\hat{P}_k-\sum_{i=0}^{d}\hat{P}_{k}^{k+i}$, we have
\begin{equation*}
\begin{aligned}
\Gamma_k&= \left[
  \begin{array}{cc}
    \displaystyle I-\hat{B}(\hat{P}_k-\sum_{i=0}^{d}\hat{P}_{k}^{k+i}) & \displaystyle-\frac{1}{\delta}\hat{C}\hat{P}_{k} \\
    \displaystyle -\delta\bar{B}(\hat{P}_k-\sum_{i=0}^{d}\hat{P}_{k}^{k+i}) & I-\bar{C}\hat{P}_{k} \\
  \end{array}
\right]
=
 \left[
  \begin{array}{cc}
    \displaystyle I-\delta BS_k & \displaystyle -C\hat{P}_{k} \\
    \displaystyle -\delta\bar{B}S_k & I-\bar{C}\hat{P}_{k} \\
  \end{array}
\right].
\end{aligned}
\end{equation*}

Through the calculation, we can obtain the reverse of $\Gamma_k$ as follows:
\begin{equation*}
\begin{aligned}
&\Gamma_k^{-1}= \left[
  \begin{array}{cc}
    \displaystyle I-\delta BS_k & \displaystyle -C\hat{P}_{k} \\
    \displaystyle -\delta\bar{B}S_k & I-\bar{C}\hat{P}_{k} \\
  \end{array}
\right]^{-1}
=
 \left[
  \begin{array}{cc}
     M_{1} & M_{2} \\
     M_{3} & M_{4} \\
  \end{array}
\right],
\end{aligned}
\end{equation*}
where
\begin{equation*}
\begin{aligned}
M_1&=\Phi,\\
M_2&=\Phi C\hat{P}_k(I-\bar{C}\hat{P}_k)^{-1},\\
M_3&=\delta(I-\bar{C}\hat{P}_k)^{-1}\bar{B}S_k\Phi,\\
M_4&=(I-\bar{C}\hat{P}_k)^{-1}+\delta(I-\bar{C}\hat{P}_k)^{-1}\bar{B}S_k\Phi CP_{k}(I-\bar{C}\hat{P}_k)^{-1},\\
\Phi&=[I-\delta BS_k-\delta C\hat{P}_k(I-\bar{C}\hat{P}_k)^{-1}\bar{B}S_k]^{-1}.
\end{aligned}
\end{equation*}

Next, by replacing $\hat{A}=I+\delta A,\hat{D}=I+\delta D,\hat{B}=\delta B,\hat{C}=\delta C$ in (\ref{DFSDE}), we can obtain that (\ref{DFSDE}) has the following form:
\begin{equation*}
\begin{aligned}
x_k=&(I+\delta A+\Delta w_{k-1}\bar{A})x_{k-1} \\
&+\left[
\begin{array}{cc}
\displaystyle B_{k-1}S_k & \displaystyle \frac{1}{\delta}C_{k-1}\hat{P}_{k} \\
\end{array}
\right]
\Gamma_{k}^{-1}
\left[
   \begin{array}{c}
     I+\delta A \\
     \delta\bar{A} \\
   \end{array}
 \right]
E(x_k|\mathcal{F}_{k-d-2})\\
=&(I+\delta A+\Delta w_{k-1}\bar{A})x_{k-1}+\big[(B_{k-1}S_kM_1+\frac{1}{\delta}C_{k-1}\hat{P}_{k}M_3)(I+\delta A) \\
&+(B_{k-1}S_kM_2+\frac{1}{\delta}C_{k-1}\hat{P}_{k}M_4)\delta\bar{A} \big] E(x_k|\mathcal{F}_{k-d-2}),
\end{aligned}
\end{equation*}
which gives that
\begin{equation*}
\begin{aligned}
x_k&=(I+\delta A+\Delta w_{k-1}\bar{A})x_{k-1}+\Big\{[B_{k-1}S_k\Phi+C_{k-1}\hat{P}_{k}(I-\bar{C}\hat{P}_k)^{-1}\bar{B}S_k\Phi](I+\delta A)\\
&+\delta B_{k-1}S_k\Phi C\hat{P}_k(I-\bar{C}P_k)^{-1}\bar{A}+C_{k-1}\hat{P}_{k}(I-\bar{C}\hat{P}_k)^{-1}\bar{A}+\delta C_{k-1}\hat{P}_{k}(I-\bar{C}P_k)^{-1}\\
&\times\bar{B}S_k\Phi C\hat{P}_{k}(I-\bar{C}\hat{P}_k)^{-1}\bar{A}\Big\}E(x_{k-1}|\mathcal{F}_{k-d-2}).
\end{aligned}
\end{equation*}
Then combining $B_{k}=\delta B+\Delta w_{k-1}\bar{B}$ and $C_{k}=\delta C+\Delta w_{k-1}\bar{C}$, we thus have
\begin{equation*}
\begin{aligned}
x_k&=(I+\delta A+\Delta w_{k-1}\bar{A})x_{k-1}+\delta\Big\{[ BS_k\Phi+ C\hat{P}_{k}(I-\bar{C}\hat{P}_k)^{-1}\bar{B}S_k\Phi](I+\delta A)\\
&\ + \delta BS_k\Phi C\hat{P}_k(I-\bar{C}\hat{P}_k)^{-1}\bar{A}+\ C\hat{P}_{k}(I-\bar{C}\hat{P}_k)^{-1}\bar{A}+ \delta C\hat{P}_{k}(I-\bar{C}\hat{P}_k)^{-1}\\
&\ \times\bar{B}S_k\Phi C\hat{P}_{k}(I-\bar{C}\hat{P}_k)^{-1}\bar{A}\Big\}E(x_{k-1}|\mathcal{F}_{k-d-2})\\
&\ +\Big\{[\Delta w_{k-1}\bar{B}S_k\Phi+\Delta w_{k-1}\bar{C}\hat{P}_{k}(I-\bar{C}\hat{P}_k)^{-1}\bar{B}S_k\Phi](I+\delta A)\\
&\ +\delta \Delta w_{k-1}\bar{B}S_k\Phi C\hat{P}_k(I-\bar{C}\hat{P}_k)^{-1}\bar{A}+\Delta w_{k-1}\bar{C}\hat{P}_{k}(I-\bar{C}\hat{P}_k)^{-1}\bar{A}\\
&\ +\delta \Delta w_{k-1}\bar{C}\hat{P}_{k}(I-\bar{C}\hat{P}_k)^{-1}\bar{B}S_k\Phi C\hat{P}_{k}(I-\bar{C}\hat{P}_k)^{-1}\bar{A}\Big\}E(x_{k-1}|\mathcal{F}_{k-d-2}).
\end{aligned}
\end{equation*}

Noticing the definition we given at the beginning, i.e. $\delta=t_{k+1}-t_k$, thus we keep the $\delta$ which outside the braces remain. Then let other $\delta\to 0$, we can obtain the continuous-time form as follows:
\begin{equation*}
\begin{aligned}
dx(t)=&\Big\{Ax(t)+\Big[BS(t)+CP(t)[I-\bar{C}P(t)]^{-1}[\bar{A}+\bar{B}S(t)]\Big]E[x(t)|\mathcal{F}_{t-h}]\Big\}dt\\
&+\Big\{\bar{A}x(t)+\Big[\bar{B}S(t)+\bar{C}P(t)[I-\bar{C}P(t)]^{-1}[\bar{A}+\bar{B}S(t)]\Big]E[x(t)|\mathcal{F}_{t-h}]\Big\}dw(t).
\end{aligned}
\end{equation*}
which is same as (\ref{x(t)}).

We are now in the position to give the proof of the Riccati equations (\ref{P(t)})-(\ref{P(t,s)}). Actually they can be obtained by taking limit to the discrete-time form (\ref{Pk})-(\ref{Pkk}).

Using the definition $\hat{A}=I+\delta A,\hat{D}=I+\delta D,\hat{Q}=\delta Q$ in $\hat{P}_k$, we can rewrite $\hat{P}_k$ as follows:
\begin{equation*}
\begin{aligned}
\hat{P}_k&=\hat{D}\hat{P}_{k+1}\hat{A}+\delta\bar{D}\hat{P}_{k+1}\bar{A}-\hat{D}\hat{P}_{k+1}^{k+d+1}\hat{A}+\hat{Q}\\
&=(I+\delta D)\hat{P}_{k+1}(I+\delta A)+\delta\bar{D}'\hat{P}_{k+1}\bar{A}-(I+\delta D)\hat{P}_{k+1}^{k+d+1}(I+\delta A)+\delta Q.\\
\end{aligned}
\end{equation*}
Through the simple calculation we have
\begin{equation*}\label{deltaP}
\begin{aligned}
\frac{\hat{P}_k-\hat{P}_{k+1}}{\delta}&=D\hat{P}_{k+1}+\hat{P}_{k+1}A+\delta D\hat{P}_{k+1}A+\bar{D}'\hat{P}_{k+1}\bar{A}-\frac{1}{\delta}\hat{P}_{k+1}^{k+d+1}\\
&\ -D\hat{P}_{k+1}^{k+d+1}-\hat{P}_{k+1}^{k+d+1}A-\delta D\hat{P}_{k+1}^{k+d+1}A+Q.
\end{aligned}
\end{equation*}
Noticing that $\hat{P}_{k+1}^{k+d+1}$ included a $\delta$, thus the above equation actually has the following form:
\begin{equation*}
\begin{aligned}
\frac{\hat{P}_k-\hat{P}_{k+1}}{\delta}&=D\hat{P}_{k+1}+\hat{P}_{k+1}A+\delta D\hat{P}_{k+1}A+\bar{D} \hat{P}_{k+1}\bar{A}-\hat{P}_{k+1}^{k+d+1}-\delta D\hat{P}_{k+1}^{k+d+1}\\
&\ -\delta \hat{P}_{k+1}^{k+d+1}A-\delta^2 D\hat{P}_{k+1}^{k+d+1}A+Q,
\end{aligned}
\end{equation*}
then letting $\delta\to 0$ we can obtain the continuous-time form as
\begin{equation*}
\begin{aligned}
-\dot{P}(t)&=DP(t)+P(t)A+\bar{D} P(t)\bar{A}+Q-P(t,t+h).
\end{aligned}
\end{equation*}
By the definition, when $\delta\to 0$, $\displaystyle\frac{1}{\delta}(\hat{P}_k-\hat{P}_{k+1})$ is defined as $-\dot{P}(t)$, and $\hat{P}_k$ becomes as $P(t)$, so the proof of convergence is no need to include here. Thus we get (\ref{P(t)}). Additionally, we verify that the derived solution indeed satisfies the D-FBSDEs by using the knowledge of the It\^{o} integral, the detailed derivations are given in Appendices. And using the same method, $P(t,t)$ can be obtained directly. Then we will illustrate how to get the continuous-time form of $\hat{P}_{k}^{k+i}$.

Similarly, by the definition of $\hat{A}$ and $\hat{D}$, we have
\begin{equation*}
\begin{aligned}
\hat{P}_{k}^{k+i}&=\hat{D}\hat{P}_{k+1}^{k+i}\hat{A}=(I+\delta D)\hat{P}_{k+1}^{k+i}(I+\delta A)\\
&=\hat{P}_{k+1}^{k+i}+\delta \hat{P}_{k+1}^{k+i}A+ \delta D\hat{P}_{k+1}^{k+i}+\delta^2 D\hat{P}_{k+1}^{k+i}A,
\end{aligned}
\end{equation*}
thus we can obtain the following relationship:
\begin{equation*}
\frac{\hat{P}_{k}^{k+i}-\hat{P}_{k+1}^{k+i}}{\delta}=\hat{P}_{k+1}^{k+i}A+ D\hat{P}_{k+1}^{k+i}+\delta D\hat{P}_{k+1}^{k+i}A.
\end{equation*}
Noticing that $\hat{P}_{k}^{k+i}$ included a $\delta$, then let $\delta\to 0$ we have
\begin{equation*}
\frac{\partial P(t,t+\theta)}{\partial\theta}=DP(t,t+\theta)A,
\end{equation*}
then solving this equation we have
\begin{equation*}
P(t,s)=e^{D(t-s)}C^*e^{A(t-s)},
\end{equation*}
let $t=s$ we can obtain that $C^*=P(s,s)$, thus we have
\begin{equation*}
P(t,s)=e^{D(s-t)}P(s,s)e^{A(s-t)}.
\end{equation*}
So we prove (\ref{P(t)})-(\ref{P(t,s)}).

Parallel to Lemma \ref{lemma2}, it is obtained that if Riccati equations (\ref{P(t)})-(\ref{P(t,s)}) have the solution such that the matrix $I-\bar C P(t)$ is invertible, then the unique solution of D-FBSDEs (\ref{FBSDE}) is given by (\ref{pt})-(\ref{qt}). The proof is now completed.
\end{proof}
\begin{remark}\label{remark5}
$E[q(t)|\mathcal{F}_{t-h}]$ can be calculated easily by taking expectation on both sides of (\ref{qt}), i.e.
\begin{equation*}
\begin{aligned}
E[q(t)|\mathcal{F}_{t-h}]&=P(t)\bar{A}E[x(t)|\mathcal{F}_{t-h}]+P(t)\bar{B}E[p(t)|\mathcal{F}_{t-h}]\\
&\ +P(t)\bar{C}E[q(t)|\mathcal{F}_{t-h}],
\end{aligned}
\end{equation*}
which gives that
\begin{equation*}
\begin{aligned}
E[q(t)|\mathcal{F}_{t-h}]&=[I-\bar{C}P(t)]^{-1}\Big\{P(t)\bar{A}E[x(t)|\mathcal{F}_{t-h}]+P(t)\bar{B}E[p(t)|\mathcal{F}_{t-h}]\Big\}.
\end{aligned}
\end{equation*}
\end{remark}

\begin{remark}\label{remark6}
The matrices of the linear D-FBSDEs can also be time-varying, the solution can be obtained similarly to Theorem \ref{theorem1} in the paper. Considering the time-varying matrices case as follows:
\begin{equation*}
\begin{aligned}
dx(t)&=\Big\{A(t)x(t)+B(t)E[p(t)|\mathcal{F}_{t-h}]+C(t)E[q(t)|\mathcal{F}_{t-h}]\Big\}dt\\
&\ \ +\Big\{\bar{A}(t)x(t)+\bar{B}(t)E[p(t)|\mathcal{F}_{t-h}]+\bar{C}(t)E[q(t)|\mathcal{F}_{t-h}]\Big\}dw(t),\\
dp(t)&=-[A'(t)p(t)+\bar{A}'(t)q(t)+Q(t)x(t)]dt+q(t)dw(t),\\
x(0)&=x_0,\ p(T)=P(T)x(T),
\end{aligned}
\end{equation*}
Then defining
\begin{eqnarray*}
A_k&=&I+\delta {A}(k)+\Delta w_k \bar{A}(k)=\hat{A}(k)+\Delta w_k \bar{A}(k),\\
B_k&=&\delta {B}(k)+\Delta w_k \bar{B}(k)=\hat{B}(k)+\Delta w_k \bar{B}(k),\\
C_k&=&\delta {C}(k)+\Delta w_k \bar{C}(k)=\hat{C}(k)+\Delta w_k \bar{C}(k),\\
D_k&=&I+\delta {D}(k)+\Delta w_k \bar{D}(k)=\hat{D}(k)+\Delta w_k \bar{D}(k).
\end{eqnarray*}
we can obtain the discrete-time form like Lemma \ref{lemma1} then applying the same procedure in this paper, we can quickly obtain the results which given by
\begin{align*}
p(t)&=P(t)x(t)-\int_{0}^{h}P(t,t+\theta)E[x(t)|\mathcal{F}_{t+\theta-h}]d\theta,\  h\leq t \leq T, \\
q(t)&=P(t)\Big\{\bar{A}(t)x(t)+\bar{B}(t)E[p(t)|\mathcal{F}_{t-h}]+\bar{C}(t)E[q(t)|\mathcal{F}_{t-h}]\Big\},\\
\nonumber dx(t)&=\Big\{A(t)x(t)+\Big[B(t)S(t)+C(t)P(t)[I-\bar{C}(t)P(t)]^{-1}[\bar{A}(t)+\bar{B}(t)S(t)]\Big]\\
 &\ \times E[x(t)|\mathcal{F}_{t-h}]\Big\}dt+\Big\{\bar{A}(t)x(t)+\Big[\bar{B}(t)S(t)+\bar{C}(t)P(t)[I-\bar{C}(t)P(t)]^{-1}\\
\nonumber &\ \times[\bar{A}(t)+\bar{B}(t)S(t)]\Big]E[x(t)|\mathcal{F}_{t-h}]\Big\}dw(t).
\end{align*}
where $P(t),P(t,s)$ satifies
\begin{eqnarray*}
-\dot{P}(t)&=&D(t)P(t)+P(t)A(t)+\bar{D}(t) P(t)\bar{A}(t)+Q(t)-P(t,t+h), \\
\nonumber P(t,t)&=&-[I-\bar{C}(t)P(t)]^{-1}\Big[(S(t)B(t)+\bar{D}(t)P(t)\bar{B}(t))S(t)[I-\bar{C}(t)P(t)]\\
& &\ +(S(t)C(t)+\bar{D}(t)P(t)\bar{C}(t))P(t)[\bar{A}(t)+\bar{B}(t)S(t)]\Big],\\
P(t,s)&=&e^{D(t)(s-t)}P(s,s)e^{A(t)(s-t)},
\end{eqnarray*}
Thus the discretization method proposed in this paper can also solve the D-FBSDEs with time-varying matrices and there are no
any additional difficulties.

\end{remark}

\section{Application to Optimal Control Problems with State Transmission Delay}

In this section, we will apply the results to a kind of stochastic LQ problem with state transmission delay where the system is governed by
\begin{equation}\label{remark1}
\begin{aligned}
dx(t)&=[Ax(t)+Bu(t)]dt+[\bar{A}x(t)+\bar{B}u(t)]dw(t),\\
x(0)&=x_0,\\
\end{aligned}
\end{equation}
which $x(t)\in R^n$ is the state, $u(t)\in R^m$ is the control input, $w(t)$ is the one-dimensional standard Brownian motion, $x_0$ is prescribed initial value, and $A,B,\bar{A},\bar{B}$ are constant matrices with compatible dimensions. It is assumed that there exists state transmission delay $h$ from plant to controller, i.e., only $x(t-s), h<s$, is available at time $t$ to design feedback controller $u(t)$. The associated cost function is given as follows:
\begin{equation}\label{remark2}
\begin{aligned}
J_T&=E\Big[\int_{0}^{T}x'(t)Qx(t)dt+\int_{0}^{T}u'(t)Ru(t)dt+x'(T)Hx(T)\Big],
\end{aligned}
\end{equation}
where $Q,R,H$ are positive semi-definite matrices of compatible dimensions. The related problem is given as follows:

The Problem is to find the causal and $\mathcal{F}_{t-h}$-adapted optimal control $u(t)$ to minimize (\ref{remark2}) subject to (\ref{remark1}).

Following the similar arguments as in \cite{Juan2}, we know that the problem is uniquely solvable if and only if there exists a unique $(x(t), p(t), q(t), u(t))$ satisfying the following FBSDEs:
\begin{equation}\label{remark3}
\left\{
\begin{aligned}
dx(t)&=[Ax(t)+Bu(t)]dt+[\bar{A}x(t)+\bar{B}u(t)]dw(t),\\
dp(t)&=-[A'p(t)+\bar{A}'q(t)+Qx(t)]dt+q(t)dw(t),\\
0&=Ru(t)+E[B'p(t)+\bar{B}'q(t)|\mathcal{F}_{t-h}],\\
x(0)&=x_0,\\
p(T)&=Hx(T),
\end{aligned}
\right.
\end{equation}
while $H$ is defined in (\ref{remark2}) and $x(T)$ is the state at time $T$.

It is clearly to see that the core of solving the problem is to solve the D-FBSDEs (\ref{remark3}), which is same as what we considered in this paper. Thus we can obtain the following results by applying the Theorem \ref{theorem1} directly:

\begin{coro}
The problem is uniquely solvable if the following Riccati equations admit a solution such that the matrix $R(t)$ is strictly positive definite:
\begin{eqnarray}
-\dot{P}(t)&=&A'P(t)+P(t)A+\bar{A}'P(t)\bar{A}+Q-P(t,t+h), \label{35}\\
P(t,t)&=&M'(t)R^{-1}(t)M(t),\\
R(t)&=&\bar{B}'P(t)\bar{B}+R,\label{R(t)}\\
M(t)&=&B'P(t)+\bar{B}'P(t)\bar{A}-\int_{0}^{h}B'P(t,s)ds,\label{M(t)}\\
P(t,s)&=&e^{A'(s-t)}P(s,s)e^{A(s-t)}.\label{39}
\end{eqnarray}
In this case, the optimal controller is given by
\begin{equation}
u(t)=-R^{-1}(t)M(t)\hat{x}(t|t-h),\label{opc}
\end{equation}
where $\hat{x}(t|t-h)$ is calculated by
\begin{equation*}
\hat{x}(t|t-h)=
\left\{
\begin{aligned}
&e^{Ah}x(t-h)+\int_{t-h}^{t}e^{A(t-\theta)}Bu(\theta)d\theta,\ t\geq h;\\
&e^{At}x(0)+\int_{0}^{t}e^{A(t-\theta)}Bu(\theta)d\theta,\ t< h;
\end{aligned}
\right.
\end{equation*}
and the solution to D-FBSDEs (\ref{remark3}) is given by
\begin{align}
p(t)&=P(t)x(t)-\int_{0}^{h}P(t,t+\theta)\hat{x}(t|t+\theta-h)d\theta, \label{41}\\
q(t)&=P(t)[\bar{A}x(t)+\bar{B}u(t)],\label{42}
\end{align}
where $\hat{x}(t|s)=E[x(t)|\mathcal{F}_s]$.
\end{coro}

\begin{proof}
Through (\ref{remark3}) and under the condition that $R>0$, we can rewrite (\ref{remark3}) as follows:
\begin{equation}\label{remark4}
\left\{
\begin{aligned}
dp(t)&=-[A'p(t)+\bar{A}'q(t)+Qx(t)]dt+q(t)dw(t),\\
dx(t)&=\Big\{Ax(t)-BR^{-1}E[B'p(t)+\bar{B}'q(t)|\mathcal{F}_{t-h}]\Big\}dt\\
&\ +\Big\{\bar{A}x(t)-\bar{B}R^{-1}E[B'p(t)+\bar{B}'q(t)|\mathcal{F}_{t-h}]\Big\}dw(t),\\
x(0)&=x_0,\\
p(T)&=P(T)x(T).
\end{aligned}
\right.
\end{equation}
which can be seen as the special case of (\ref{FBSDE}). i.e. $D=A'$, $\bar{D}=\bar{A}'$ and $B,\bar{B},C,\bar{C}$ are replaced by $-BR^{-1}B',-\bar{B}R^{-1}B',-BR^{-1}\bar{B}'$ and $-\bar{B}R^{-1}\bar{B}'$.

In this situation, the related Riccati equations (\ref{P(t)})-(\ref{P(t,s)}) are become as (\ref{35})-(\ref{39}). Then introducing these specific matrices in Theorem \ref{theorem1}, we have:
\begin{align*}
p(t)=P(t)x(t)-\int_{0}^{h}P(t,t+\theta)\hat{x}(t\mid t+\theta-h)d\theta,
\end{align*}
and
\begin{eqnarray*}
q(t)&=&P(t)\Big\{\bar{A}x(t)-\bar{B}R^{-1}B'E[p(t)|\mathcal{F}_{t-h}]-\bar{B}R^{-1}\bar{B}'E[q(t)|\mathcal{F}_{t-h}]\Big\},
\end{eqnarray*}
combining the equilibrium condition in (\ref{remark3}), i.e, $u(t)=-R^{-1}E[B'p(t)+\bar{B}'q(t)|\mathcal{F}_{t-h}]$, we have
\begin{eqnarray*}
q(t)&=&P(t)[\bar{A}x(t)+\bar{B}u(t)],
\end{eqnarray*}
where $P(t),P(t,s)$ satisfies (\ref{35})-(\ref{39}). Thus we have verified (\ref{41}) and (\ref{42}).

Next combining $p(t),q(t)$ and the equilibrium condition, we can obtain that
\begin{equation*}
\begin{aligned}
0&=Ru(t)+E[B'p(t)+\bar{B}'q(t)|\mathcal{F}_{t-h}]\\
&=Ru(t)+E\Big\{B'[P(t)x(t)-\int_{0}^{h}P(t,t+\theta)\hat{x}(t|t+\theta-h)d\theta]\\
&\ +\bar{B}'P(t)[\bar{A}x(t)+\bar{B}u(t)]|\mathcal{F}_{t-h}\Big\},
\end{aligned}
\end{equation*}
then combining (\ref{R(t)}) and (\ref{M(t)}), it gives that
\begin{equation*}
u(t)=-R^{-1}(t)M(t)\hat{x}(t|t-h).
\end{equation*}
Thus we can obtain that the optimal controller is as (\ref{opc}). This completes the proof.
\end{proof}
\begin{remark}
Using the method of discretizaiton, we present a different way which is more intuitive to solve this FBSDEs. And it can be seen that solving the D-FBSDEs is useful for LQ control problems. Besides, this method is not only for the single input delay, but can be also applied in the problem with multiple input delay.
\end{remark}

\section{Conclusions}
\label{sec:conclusions}
In this paper, we studied the solution of the continuous-time D-FBSDEs, and the solution has been obtained by using the method of discretization. First, we transform the continuous-time D-FBSDEs into the discrete-time form, in this situation, we can obtain the solution of this discrete-time form. Next we obtain the solution of the continuous-time D-FBSDEs by limiting the results of discrete-time case.

It is believed that the proposed method of discretization is powerful in deriving the explicit solution for delayed FBSDEs in particularly the general case, and thus provides complete solution to the complicated optimal control problem for linear stochastic delayed systems.

\section{Appendices}
In this section, we verify that the derived solution (\ref{pt})-(\ref{x(t)}) indeed satisfies the D-FBSDEs (\ref{FBSDE}) by using the continuous-time technique. The detailed derivations are given as follows.

First, we define
\begin{align}
\hat{p}(t)&=P(t)x(t)-\int_{0}^{h}P(t,t+\theta)E[x(t)|\mathcal{F}_{t+\theta-h}]d\theta,\label{hpt}\\ \hat{q}(t)&=P(t)\Big\{\bar{A}x(t)+\bar{B}E[\hat{p}(t)|\mathcal{F}_{t-h}]+\bar{C}E[\hat{q}(t)|\mathcal{F}_{t-h}]\Big\},\label{hqt}
\end{align}
where $x(t)$ is defined by (\ref{x(t)}). Using the same procedure introduced in Remark \ref{remark5}, we can obtain that
\begin{align}
E[\hat{q}(t)|\mathcal{F}_{t-h}]=[I-\bar{C}P(t)]^{-1}\Big\{P(t)\bar{A}E[x(t)|\mathcal{F}_{t-h}]+P(t)\bar{B}E[\hat{p}(t)|\mathcal{F}_{t-h}]\Big\}.\label{ehqt}
\end{align}
Then we will verify that (\ref{pt})-(\ref{x(t)}) is the solution to (\ref{FBSDE}). Considering (\ref{x(t)}) with $\displaystyle S(t)=P(t)-\int_{0}^{h}P(t,t+\theta)d\theta$, we have
\begin{align*}
dx(t)&=\Big\{Ax(t)+\Big[B(P(t)-\int_{0}^{h}P(t,t+\theta)d\theta)+CP(t)[I-\bar{C}P(t)]^{-1}[\bar{A}\\
&\ +\bar{B}(P(t)-\int_{0}^{h}P(t,t+\theta)d\theta)]\Big]E[x(t)|\mathcal{F}_{t-h}]\Big\}dt\\
&\ +\Big\{\bar{A}x(t)+\Big[\bar{B}(P(t)-\int_{0}^{h}P(t,t+\theta)d\theta)+\bar{C}P(t)[I-\bar{C}P(t)]^{-1}\\
&\ \times[\bar{A}+\bar{B}(P(t)-\int_{0}^{h}P(t,t+\theta)d\theta)]\Big]E[x(t)|\mathcal{F}_{t-h}]\Big\}dw(t).
\end{align*}
This gives that
\begin{align*}
dx(t)&=\Big\{Ax(t)+BE[P(t)x(t)-\int_{0}^{h}P(t,t+\theta)E[x(t)|\mathcal{F}_{t+\theta-h}]d\theta|\mathcal{F}_{t-h}]\\
&+C[I-\bar{C}P(t)]^{-1}\Big[P(t)\bar{A}E[x(t)|\mathcal{F}_{t-h}]+P(t)\bar{B}E[P(t)x(t)\\
&-\int_{0}^{h}P(t,t+\theta)E[x(t)|\mathcal{F}_{t+\theta-h}]d\theta|\mathcal{F}_{t-h}]\Big]\Big\}dt\\
&+\Big\{\bar{A}x(t)+\bar{B}E[P(t)x(t)-\int_{0}^{h}P(t,t+\theta)E[x(t)|\mathcal{F}_{t+\theta-h}]d\theta|\mathcal{F}_{t-h}]\\
&+\bar{C}[I-\bar{C}P(t)]^{-1}\Big[P(t)\bar{A}E[x(t)|\mathcal{F}_{t-h}]+P(t)\bar{B}E[P(t)x(t)\\
&-\int_{0}^{h}P(t,t+\theta)E[x(t)|\mathcal{F}_{t+\theta-h}]d\theta|\mathcal{F}_{t-h}]\Big]\Big\}dw(t).\\
\end{align*}
By combining (\ref{hpt}) and (\ref{ehqt}), we obtain that
\begin{equation}\label{dhxt}
\begin{aligned}
dx(t)&=\Big\{Ax(t)+BE[\hat{p}(t)|\mathcal{F}_{t-h}]+CE[\hat{q}(t)|\mathcal{F}_{t-h}]\Big\}dt\\
&\ \ +\Big\{\bar{A}x(t)+\bar{B}E[\hat{p}|\mathcal{F}_{t-h}]+\bar{C}E[\hat{q}(t)|\mathcal{F}_{t-h}]\Big\}dw(t).\\
\end{aligned}
\end{equation}

Next, from (\ref{dhxt}) and (\ref{P(t)})-(\ref{P(t,s)}), by applying It\^o formula to $d[P(t)x(t)]$, we have
\begin{align}
d[P(t)x(t)]&=dP(t)x(t)+P(t)dx(t)+dP(t)dx(t)\nonumber\\
&=-\Big[DP(t)+P(t)A+\bar{D}P(t)\bar{A}+Q-P(t,t+h)\Big]x(t)\nonumber\\
&\ +P(t)\Big\{Ax(t)+BE[\hat{p}(t)|\mathcal{F}_{t-h}]+CE[\hat{q}(t)|\mathcal{F}_{t-h}]\Big\}dt\nonumber\\
&\ +P(t)\Big\{\bar{A}x(t)+\bar{B}E[\hat{p}(t)|\mathcal{F}_{t-h}]+\bar{C}E[\hat{q}(t)|\mathcal{F}_{t-h}]\Big\}dw(t),\label{a1}
\end{align}
by using the definition $\hat{x}(t|s)=E[x(t)|\mathcal{F}_s]$, $\displaystyle d\int_{0}^{h}P(t,t+\theta)E[x(t)|\mathcal{F}_{t+\theta-h}]d\theta$ is calculated as
\begin{align*}
&d\int_{0}^{h}P(t,t+\theta)E[x(t)|\mathcal{F}_{t+\theta-h}]d\theta\\
&=d\int_{t}^{t+h}P(t,s)\hat{x}(t|s-h)ds\\
&=\int_{t}^{t+h}\partial_t [P(t,s)\hat{x}(t|s-h)]ds+P(t,t+h)x(t)-P(t,t)\hat{x}(t|t-h)\\
&=\int_{t}^{t+h}[\partial_t P(t,s)]\hat{x}(t|s-h)ds+\int_{t}^{t+h}P(t,s)\partial _t\hat{x}(t|s-h)ds\\
&~~~ +P(t,t+h)x(t)-P(t,t)\hat{x}(t|t-h).
\end{align*}
Noting that
\begin{align*}
\int_{t}^{t+h}[\partial _tP(t,s)]\hat{x}(t|s-h)ds=\int_{t}^{t+h}\Big(\partial _t[e^{D(s-t)}P(s,s)e^{A(s-t)}]\Big)\hat{x}(t|s-h)ds,
\end{align*}
where $P(t,s)=e^{D(s-t)}P(s,s)e^{A(s-t)}$ has been used,
we thus have
\begin{align}
&d\int_{0}^{h}P(t,t+\theta)E[x(t)|\mathcal{F}_{t+\theta-h}]d\theta\nonumber\\
&=-D\Bigg[\int_{t}^{t+h}P(t,s)\hat{x}(t|s-h)ds\Bigg]dt+\Bigg[\int_{t}^{t+h}P(t,s)ds\Bigg]\Big[BE[\hat{p}(t)|\mathcal{F}_{t-h}]\nonumber\\
&~~~+CE[\hat{q}(t)|\mathcal{F}_{t-h}\Big]dt+P(t,t+h)x(t)-P(t,t)\hat{x}(t|t-h).\label{a2}
\end{align}
From (\ref{a1}) and (\ref{a2}), it is obtained that
\begin{align*}
&d[P(t)x(t)-\int_{0}^{h}P(t,t+\theta)E[x(t)|\mathcal{F}_{t+\theta-h}]d\theta]\\
=&-\Big[DP(t)+P(t)A+\bar{D}P(t)\bar{A}+Q-P(t,t+h)\Big]x(t)dt\\
&+P(t)\Big\{Ax(t)+BE[\hat{p}(t)|\mathcal{F}_{t-h}]+CE[\hat{q}(t)|\mathcal{F}_{t-h}]\Big\}dt\\
&+P(t)\Big\{\bar{A}x(t)+\bar{B}E[\hat{p}(t)|\mathcal{F}_{t-h}]+\bar{C}E[\hat{q}(t)|\mathcal{F}_{t-h}]\Big\}dw(t)\\ &+D\Bigg[\int_{t}^{t+h}P(t,s)\hat{x}(t|s-h)ds\Bigg]dt\\
&-\Bigg[\int_{t}^{t+h}P(t,s)ds\Bigg]\Big[BE[\hat{p}(t)|\mathcal{F}_{t-h}]+CE[\hat{q}(t)|\mathcal{F}_{t-h}\Big]dt-P(t,t+h)x(t)dt\\
&+P(t,t)\hat{x}(t|t-h)dt.\\
=&-D\Bigg[P(t)x(t)-\int_{t}^{t+h}P(t,s)\hat{x}(t|s-h)ds\Bigg]dt-Qx(t)dt\\
&-\bar{D}\Big\{P(t)\big[\bar{A}x(t)+\bar{B}E[\hat{p}(t)|\mathcal{F}_{t-h}]+\bar{C}E[\hat{q}(t)|\mathcal{F}_{t-h}]\big]\Big\}dt\\
&+P(t)\Big\{\bar{A}x(t)+\bar{B}E[\hat{p}(t)|\mathcal{F}_{t-h}]+\bar{C}E[\hat{q}(t)|\mathcal{F}_{t-h}]\Big\}dw(t)\\
&+\Big\{(P(t)B+\bar{D}P(t)\bar{B})E[\hat{p}(t)|\mathcal{F}_{t-h}]+(P(t)C+\bar{D}P(t)\bar{C})E[\hat{q}(t)|\mathcal{F}_{t-h}]\Big\}dt\\
&-\Bigg[\int_{t}^{t+h}P(t,s)ds\Bigg]\Big[BE[\hat{p}(t)|\mathcal{F}_{t-h}]+CE[\hat{q}(t)|\mathcal{F}_{t-h}\Big]dt+P(t,t)\hat{x}(t|t-h)dt.
\end{align*}
Then combining $\displaystyle S(t)=P(t)-\int_{0}^{h}P(t,t+\theta)d\theta$ and $P(t,t)$ into the above equation, it follows that
\begin{align}\label{dhpt}
d\hat{p}(t)&=\nonumber d[P(t)x(t)-\int_{0}^{h}P(t,t+\theta)E[x(t)|\mathcal{F}_{t+\theta-h}]d\theta]\\
&=-[D\hat{p}(t)+\bar{D}\hat{q}(t)+Qx(t)]dt+\hat{q}(t)dw(t),
\end{align}
and
\begin{align}\label{dhqt}
\hat{q}(t)&=P(t)\Big\{\bar{A}x(t)+\bar{B}E[\hat{p}(t)|\mathcal{F}_{t-h}]+\bar{C}E[\hat{q}(t)|\mathcal{F}_{t-h}]\Big\}.
\end{align}

By making comparison between (\ref{dhxt}), (\ref{dhpt}), (\ref{dhqt}) and (\ref{FBSDE}) and combining with the same terminal condition, it is obtained that
(\ref{dhxt}), (\ref{dhpt}), (\ref{dhqt}) are the same as (\ref{FBSDE}), that is, (\ref{pt})-(\ref{x(t)}) is the solution to the D-FBSDEs (\ref{FBSDE}). This completes the proof.

\end{document}